\documentclass[a4paper]{amsart}[12pt]

\usepackage{color}
\usepackage{amsxtra}
\usepackage{mathtools}
\usepackage{enumerate}
\usepackage{nicefrac}
\usepackage{psfrag}
\usepackage[normalem]{ulem}

\usepackage{ mathrsfs }
\usepackage{tikz}
\usetikzlibrary{shadings}
\usepackage{color}
\usepackage{hyperref}
\hypersetup{
  colorlinks   = true, 
  urlcolor     = blue, 
  linkcolor    = blue, 
  citecolor   = red 
}

\newtheorem{theorem}{Theorem}[section]
\newtheorem{lemma}{Lemma}[section]

\newtheorem{corollary}{Corollary}
\newtheorem{remark}{Remark}[section]

\newcommand{\R}{\mathbb{R}}
\newcommand{\De}{\Delta}
\newcommand{\de}{\delta}
\newcommand{\ds}{\displaystyle}
\newcommand{\al}{\alpha}
\newcommand{\la}{\lambda}
\newcommand{\e}{\varepsilon}

\title[Periodic solutions for the fractional Laplacian] {
Periodic solutions for the one-dimensional fractional Laplacian
}

\author[B. Barrios]{B. Barrios}
	\address{B. Barrios \hfill\break\indent
		Departamento de An\'{a}lisis Matem\'{a}tico,
		Universidad de La Laguna\hfill 
		\break \indent C/. Astrof\'{\i}sico Francisco S\'{a}nchez s/n, 
		38200 -- La Laguna, SPAIN}
	\email{bbarrios@ull.es}
	
\author[J. Garc\'{\i}a-Meli\'{a}n]{J. Garc\'{\i}a-Meli\'{a}n}
 	\address{J. Garc\'{\i}a-Meli\'{a}n \hfill\break\indent
		Departamento de An\'{a}lisis Matem\'{a}tico, 
		Universidad de La Laguna
		\hfill \break \indent C/. Astrof\'{\i}sico 
		Francisco S\'{a}nchez s/n, 38200 -- La Laguna, SPAIN
		\hfill\break\indent
		{\rm and} \hfill\break
		\indent 
		Instituto Universitario de Estudios Avanzados (IUdEA) 
		\hfill\break\indent en F\'{\i}sica At\'omica,
		Molecular y Fot\'onica, 
		\hfill\break\indent Universidad de La Laguna
		\hfill\break\indent C/. Astrof\'{\i}sico Francisco
		S\'{a}nchez s/n, 38200 -- La Laguna, SPAIN.}
	\email{jjgarmel@ull.es}

\author[A. Quaas]{A. Quaas} 
	\address{A. Quaas\hfill\break\indent
		Departamento de Matem\'{a}tica, \hfill\break\indent
		Universidad T\'ecnica Federico Santa Mar\'{\i}a
		\hfill\break\indent  Casilla V-110, Avda. Espa\~na, 
		1680 -- Valpara\'{\i}so, CHILE.}
	\email{{\tt alexander.quaas@usm.cl}}

\begin{document}

\begin{abstract}  In this paper we are concerned with the construction of periodic 
solutions of the nonlocal problem $(-\De)^s u= f(u)$ in $\R$, where $(-\De)^s$ 
stands for the $s$-Laplacian, $s\in (0,1)$. We introduce a suitable framework which 
allows, by means of regularity, to link the searching of such solutions into the existence of the ones of a semilinear
problem in a suitable Hilbert space. Then by a bifurcation theory from eigenvalues of odd multiplicity  and also variational method that avoid the 
 constant solutions we get existence theorems which are lately enlightened with the analysis of some examples. In particular, multiplicity results for generalized Benjamin-Ono equation are obtained.
\end{abstract}

\maketitle

\section{Introduction}

In the recent years the study of equations driven by 
nonlocal operators (some times termed as nonlocal diffusion equations) has been increased. The leading 
role in most of these has been played by the well-known \emph{fractional Laplacian} 
operator, which is defined on smooth functions as
\begin{equation}\label{eq-operador}
(-\Delta)^s u(x) = c(N,s) \int_{ \mathbb{R}^N} \frac{u(x)-u(y)}{|x-y|^{N+2s}} dy.
\end{equation}
Here $s\in (0,1)$, the integral has to be understood in the principal value sense and 
$C(N,s)$ is a positive constant whose value will be of no importance to us. Thus we will 
plainly omit it in most of our arguments.

One of the important motivations in the study of equations involving the fractional 
Laplacian is to test whether the known properties for its local version, the Laplacian, 
obtained when setting $s=1$, remain valid in the full range $s\in (0,1)$. Thus all sorts of 
problems related to this operator have been considered so far. 

Our interest in the present work is to deal with problems posed in the whole real line $\R$:
\begin{equation}\label{eq-problema}
(-\Delta)^s u = f(u) \quad \text{in } \mathbb{R}.
\end{equation}
Particular types of solutions have been obtained for problem \eqref{eq-problema} depending on 
the `shape' of the nonlinearity. For instance, layer solutions in \cite{MR3165278} or 
\cite{MR3280032} or ground states in \cite{MR3002595}. However, at the best of our knowledge, 
the existence of periodic solutions to \eqref{eq-problema} has not been obtained so far except 
for \cite{MR3625076}, where a very particular $f$ is considered.  
In spite of this, we refer the reader to the papers \cite{MR3645938,MR3619062,MR3635638,MR3583529,MR3694655, 2015arXiv:1504.03493v1} and the references therein, where 
a kind of periodic problem is analyzed. But it is not clear that the solutions 
obtained are indeed solutions of a problem posed in $\R$ as \eqref{eq-problema}.
It is also worthy of mention that in most of the referred works the problem at hand is 
studied with the use of the well-known local extension introduced in \cite{MR2354493}.

Most of the properties obtained in previous papers show that essentially similar phenomena as 
in the local case appears. Thus the problem 
\begin{equation}\label{eq-problema-local}
-u'' = f(u) \quad \hbox{in }\mathbb{R}
\end{equation}
is taken as a guide to be followed. Since periodic solutions for problem \eqref{eq-problema-local} 
can be constructed very easily using ODE techniques, it is to be expected that the same happens with 
problem \eqref{eq-problema} in spite of ODE analysis being not available. This will be the main 
objective of this work.

\medskip
	
Perhaps the most interesting point in the present paper is that periodic solutions of 
\eqref{eq-problema} can be obtained by solving an adequate `periodic' problem, for 
{which tools of nonlinear analysis can be applied}. Without loss of generality, we 
assume for the moment that we are searching for $2\pi$-periodic solutions. Then 
it can be seen that the fractional Laplacian of a smooth such function reduces to 	
\begin{equation}\label{eq-operador-ele}
\mathcal{L}u(x) = \int_{0}^{2\pi} (u(x) - u(y)) H(x-y) dy, \qquad x\in (0,2\pi),
\end{equation}
where 
\begin{equation}\label{eq-def-H}
H(z)= \sum_{n=-\infty}^\infty \frac{1}{|z-2\pi n|^{1+2s}}, \quad 0<|z|<2\pi
\end{equation}
(cf. Lemma \ref{lema-operador} in Section \ref{sect-periodic}). 
Thus it is important to fully understand the operator $\mathcal{L}$. In order to 
define it in a weak sense, we consider the space $X$ \label{espacio} defined as the closure of the set of 
$2\pi$-periodic functions $u\in C^1(\R)$ with the norm
$$
\| u\|:= \frac{1}{2} \int_0^{2\pi} \hspace{-2mm} \int_0^{2\pi} (u(x)-u(y))^2 H(x-y)dy dx + 
\int_0^{2\pi} u(x)^2 dx.
$$
It is standard that $X$ is a Hilbert space when provided with the inner product
\begin{equation}\label{producto_interior}
\left< u,v\right>:= \frac{1}{2} \int_0^{2\pi} \hspace{-2mm} \int_0^{2\pi} (u(x)-u(y))(v(x)-v(y)) H(x-y)dy dx + 
\int_0^{2\pi} u(x)v(x) dx.
\end{equation}
The space $X$ possesses good embedding properties which follow directly from the 
trivial relation $\| u\|_{H^s(0,2\pi)}\le \| u\|$ for every $u\in X$ (see \cite{MR2944369} for definition 
and properties of $H^s(0,2\pi)$). We will make them explicit when necessary.

The most important properties of operator $\mathcal{L}$ will be given in Sections \ref{sect-periodic} 
and \ref{sect-global}. We only mention here that our first result will link weak solutions of the 
semilinear periodic problem 
\begin{equation}\label{eq-semilineal-periodic}
\mathcal{L}u = f(u) \quad \hbox{in } (0,2\pi)
\end{equation}
with those of the global problem 
\begin{equation}\label{eq-semilineal-global}
(-\De)^s u = f(u) \quad \hbox{in } \R.
\end{equation}
The hypotheses that we need on $f$ are quite natural. We will assume that $f$ is {a locally Lipschitz $2\pi$-periodic function that has a subcritical growth at infinity}:
\begin{equation}\label{eq-hipo}
|f(t)| \le C (1+|t|)^p \quad t\in \R,
\end{equation}
where 
$$
1<p< 2^*_s-1= \frac{1+2s}{1-2s},
$$
when $s<\frac{1}{2}$ and $1<p<\infty$ when  $s\geq\frac{1}{2}$. 
Here $2^*_s:=\frac{2}{1-2s}$ stands for the usual Sobolev exponent 
for the space $H^s$. Let us mention in passing that the function $f$ could also depend on 
the variable $x$, as long as this dependence is $2\pi$-periodic and some extra regularity 
is imposed.

These hypotheses will be termed throughout as hypotheses $(H)$. 

{Now we can state one of our main theorems that relates, by regularity estimates, weak solution of \eqref{eq-semilineal-periodic} with 
 $2\pi$-periodic classical solution of \eqref{eq-semilineal-global}.}

\begin{theorem}\label{th-equivalencia}
Assume $f$ verifies hypotheses (H). Let $u\in X$ be a weak solution of problem \eqref{eq-semilineal-periodic}. 
Then $u\in C^{2s+\al}(\R)$ for every $\al\in (0,1)$, and $u$ is a $2\pi$-periodic classical solution of 
\eqref{eq-semilineal-global}
\end{theorem}

\bigskip

Once the equivalence between problems \eqref{eq-semilineal-global} and \eqref{eq-semilineal-periodic} is 
established, we wish to obtain some general existence theorems. 
To begin with, we will consider next the application of the bifurcation theory to the model 
problem 
\begin{equation}\label{eq-bifur}
(-\De)^s u = \la u +f(u) \quad \hbox{in } \R,
\end{equation}
where $f$ is a smooth function with no `linear part' at zero, {that is $f'(0)=0$}, and $\la\in \R$ is a bifurcation parameter. 
It is worthy of mention that the independence of $f$ on the parameter is important in order to be able 
to study problem \eqref{eq-bifur}. This has a simple explanation: most of the achievements on bifurcation 
theory rely on the existence of an eigenvalue of the linearized problem
\begin{equation}\label{eq-eigen-global}
(-\De)^s u  = \la u \quad \hbox{in } \R
\end{equation}
of odd multiplicity. However, as we show in Section \ref{sect-eigen}, all eigenvalues of problem 
\eqref{eq-eigen-global} aside $\la=0$ are of multiplicity two. Besides, the eigenvalue $\la=0$, 
which is simple, usually gives rise after bifurcation only constant solutions, which are 
uninteresting (cf. Remark \ref{remark-crandall-rabinowitz}). This forces us to use the results of
bifurcation for operators of variational type, 
obtained for instance in \cite{MR0312348} or \cite{MR0348570} (see also \cite{MR845785}, \cite{MR0463990}), 
which at the best of our knowledge are only valid for problems with a special structure.{The existence result regarding with \eqref{eq-bifur} is the following}

\begin{theorem}\label{th-bifurcacion}
Assume $f\in C^1(\R)$ is such that $f'(0)=0$. Then for every {$k\in\mathbb{N}$} there exists $r_0=r_0(k)$ 
such that for every $r\in (0,r_0)$ problem \eqref{eq-bifur} admits two periodic {classical} solutions 
$(\la_i,u_i)$, $i=1,2$ such that $\| u_i\|_{L^\infty(\R)}=r$ and $|\la_i-\frac{k^{2s}}{1+k^{2s}}|$ is small. 
The minimal period of these solutions is 
$$
2\pi (1-\la_i)^{-\frac{1}{2s}},
$$
therefore close to $2\pi (1+k^{2s})^\frac{1}{2s}$.
\end{theorem}

\bigskip

Our next purpose is to show how variational methods can also be applied to deal with 
the existence of weak solutions of \eqref{eq-bifur}. If we assume $f$ has a subcritical 
growth, then the functional 
$$
J(u)=\frac{1}{4} \int_0^{2\pi} \hspace{-2mm} \int_0^{2\pi} \hspace{-1mm}(u(x)-u(y))^2 H(x-y)dy dx 
-\frac{\la}{2} \int_0^{2\pi} u(x)^2 dx - \hspace{-1mm} \int_0^{2\pi} F(u(x)) dx
$$
is well-defined in $X$. Moreover, in a rather standard way, it can be checked that its critical 
points provide with weak solutions of \eqref{eq-bifur} in $X$. Therefore it is to be expected 
that, after placing some suitable restrictions on $f$, the most popular variational methods 
can be used. 

As an example, by imposing the standard Ambrosetti-Rabinowitz condition on $f$ with 
some other minor technical requirements, the proofs in \cite{MR845785} can be adapted 
to our framework. As a matter of fact, this adaptation has already been made in the 
context of a nonlocal Dirichlet problem in \cite{MR3002745} and only minor changes 
in their proofs would be needed. 

However, an additional precaution has to be taken. By applying either the Mountain 
Pass theorem or the Linking theorem it can be proved that problem \eqref{eq-bifur} 
admits a nontrivial solution, but this solution could be a constant. To avoid this 
issue we are restricting only to the special problem 
\begin{equation}\label{eq-potencia-impar}
(-\De)^s u = \la u +|u|^{p-1}u \quad \hbox{in } \R,
\end{equation}
where $p>1$ is subcritical. In this particular case, taking advantage of the homogeneity 
of the problem, solutions can be found by constrained minimization when $\la<0$, and 
they can be guaranteed to be nonconstant only if $|\la|$ is large enough. When $\la>0$, 
however, the Linking theorem has to be used. That is,

\begin{theorem}\label{th-variacional}
Let $p>1$, and if $s<\frac{1}{2}$ assume that $p<2^*_s-1$. Then there 
exists $\la_0 \ge 0$ such that, if $\la< -\la_0$ ( resp. $\la > 0$), problem \eqref{eq-potencia-impar} 
admits at least a $2\pi$-periodic nonconstant {classical} positive (resp. sign-changing) solution.
\end{theorem}
Clearly the classical negative solution of \eqref{eq-potencia-impar}  for $\la< -\la_0$ is also obtained simply by changing the sign of the positive one given by the previous theorem. 
\bigskip

To conclude the Introduction, let us mention that in Section \ref{sect-ejemplos} 
we will analyze some examples which include power-type. In particular, we will be able 
to obtain some existence theorems for periodic solutions which somehow resemble those 
in the local case $s=1$. {As a byproduct of these examples existence of two periodic solutions of a generalized Benajamin-Ono equation are obtained. Notice that one dimensional positive solution of Benajamin-Ono equation are studied in \cite{MR3070568} where other physical relevant models related to our equations are also discussed.}

\medskip

The rest of the paper is organized as follows: in Section \ref{sect-periodic} we 
introduce the periodic problem and give some of its preliminary properties. 
Section \ref{sect-global} is devoted to the solvability of a related `global' 
problem and some further regularity, which lead to the proof of Theorem \ref{th-equivalencia}. 
In Section \ref{sect-eigen}, the eigenvalue problem is considered, while 
Sections \ref{sect-bifur} and \ref{sect-variacional} are dedicated respectively to 
the proofs of Theorems \ref{th-bifurcacion} and \ref{th-variacional}. Finally we 
include some examples in Section \ref{sect-ejemplos}.

\bigskip

\section{A periodic problem}\label{sect-periodic}
\setcounter{equation}{0}

In this section, we will show first that the fractional Laplacian of a smooth, $2\pi$-periodic function 
reduces to the operator $\mathcal{L}$ defined in the Introduction. Then we will analyze 
the basic properties of this operator, concerning solvability of boundary value problems 
and regularity of solutions.

\begin{lemma}\label{lema-operador}
Assume $u\in C^{2s+\al}(\R)$ for some $\al\in (0,1)$ is $2\pi$-periodic. Then 
for every $x\in (0,2\pi)$:
$$
(-\De)^s u(x)= \mathcal{L}u(x),
$$
where $\mathcal{L}$ is given in \eqref{eq-operador-ele}.
\end{lemma}

\begin{proof}
The proof is more or less straightforward: since $u$ is sufficiently smooth its fractional 
Laplacian can be pointwise evaluated in the classical sense. Thus using the periodicity of $u$ 
we have, for $x\in (0,2\pi)$:
\begin{align*}
(-\De)^s u(x) & = \int_{-\infty}^\infty \frac{u(x)-u(y)}{|x-y|^{1+2s}} dy
= \sum_{n=-\infty}^\infty \int_{2\pi n}^{2\pi(n+1)} \frac{u(x)-u(y)}{|x-y|^{1+2s}} dy\\
& = \sum_{n=-\infty}^\infty \int_{0}^{2\pi} \frac{u(x)-u(z+2\pi n)}{|x-z-2\pi n|^{1+2s}} dz
= \sum_{n=-\infty}^\infty \int_{0}^{2\pi} \frac{u(x)-u(z)}{|x-z-2\pi n|^{1+2s}} dz.
\end{align*}
Observe that only one of the integrals in the last sum is singular and needs to be 
interpreted in the principal value sense. 
We now want to interchange the sum and the integral in the last expression. To achieve 
this it suffices to show that the series is uniformly convergent. Observe that for 
$x,z\in (0,2\pi)$ we have $|x-z|\le 2\pi$, so that for $|n|\ge 2$ we obtain
$$
\frac{|u(x)-u(z)|}{|x-z-2\pi n|^{1+2s}} 
\le 2 \| u \|_{L^\infty} (2\pi)^{-1-2s} (|n|-1)^{-1-2s}.
$$
This shows that the series above is uniformly convergent, so that we may write
$$
(-\De)^s u (x) = \int_0^{2\pi} (u(x)-u(z)) \left( \sum_{n=-\infty}^\infty \frac{1}{|x-z-2\pi n|^{1+2s}}\right) dz,
$$
and the proof is concluded.
\end{proof}

\bigskip

We want to study next some boundary value problems associated to $\mathcal{L}$. The appropriate space 
to consider such problems is the space $X$ defined in page \pageref{espacio}. We would only like to recall here that 
the obvious inequality $\| u\|_{H^s(0,2\pi)} \le \| u\|$ implies that $X$ benefits from the same 
embedding properties into $L^q$ spaces and H\"older continuous function spaces as $H^s(0,2\pi)$. 
Namely (see Theorem 6.7, Theorem 6.10 and Theorem 8.2 of \cite{MR2944369}),
$$\mbox{$X\hookrightarrow L^{2^*_s}(0,2\pi)$ if $s<\frac{1}{2}$,}$$
$$\mbox{$X \hookrightarrow L^q(0,2\pi)$ for 
every $q>1$ when $s=\frac{1}{2}$}$$
and 
$$\mbox{$X\hookrightarrow C^{0,\alpha}[0,2\pi]$ for $s>\frac{1}{2}$, $\alpha=s-\frac{1}{2}>0$.}$$
For the sake of having a well-posed problem involving $\mathcal{L}$, we will consider:
\begin{equation}\label{eq-periodic}
\mathcal{L} u + u =f \quad \hbox{in } (0,2\pi)
\end{equation}
in the weak sense, where $f\in X'$. Observe also that the above mentioned embeddings imply 
in particular that $L^q \subset X'$ when $q\ge \frac{2}{1+2s}=(2^*_s)'$.	

A solution in the weak sense is a function $u\in X$ which verifies, for every $\phi\in X$:
\begin{align*}
\frac{1}{2} \int_0^{2\pi} & \hspace{-2mm} \int_0^{2\pi} (u(x)-u(y))(\phi(x)-\phi(y)) H(x-y) dy dx\\
& + \int_0^{2\pi} u(x) \phi(x) dx = \left< f,\phi\right>_{X,X'}
\end{align*}
where $\left< ,\right>_{X,X'}$ stands for the duality pairing between $X$ and $X'$. 
This may be succinctly written as $\left< u,\phi\right>= \left< f,\phi\right>_{X,X'}$ for every 
$\phi \in X$.

Problem \eqref{eq-periodic} can now be easily solved. 

\begin{lemma}\label{lema-exist}
For every $f\in X'$, there exists a unique weak solution $u\in X$ 
of problem \eqref{eq-periodic}. Moreover, 
\begin{equation}\label{eq-estim-1}
\|u \| = \| f\|_{X'}.
\end{equation}
In addition, if $f\ge 0$ in the sense of $X'$, then $u\ge 0$ almost everywhere.
\end{lemma}

\begin{proof}
The existence, uniqueness and equality \eqref{eq-estim-1} are a direct 
consequence of Riesz representation theorem. Thus only the maximum principle 
has to be shown. By replacing $u$, $f$ with $-u$ and $-f$, we may assume 
that $f\le 0$ instead. This means
\begin{equation}\label{eq-aux-1}
\frac{1}{2} \int_0^{2\pi} \hspace{-2mm} \int_0^{2\pi} (u(x)-u(y))(\phi(x)-\phi(y)) 
H(x-y) dy dx + \int_0^{2\pi} u(x) \phi(x) dx\le 0
\end{equation}
for every nonnegative $\phi \in X$. It is not hard to see that $\phi=u^+=\max\{u,0\}$ is a 
valid choice, and using 
$$
(u(x)-u(y))(u^+(x)-u^+(y)) \ge (u^+(x)-u^+(y))^2
$$
for $x,y \in [0,2\pi]$ we obtain from \eqref{eq-aux-1} that $\| u^+\|=0$, hence 
$u\le 0$ a. e., as was to be shown.
\end{proof}

\bigskip

Lemma \ref{lema-exist} allows to define a solution operator $K: X'\to X$, where for every 
$f\in X'$ we denote $Kf=u$. In some situations, it is more convenient to consider this operator 
as
\begin{equation}\label{eq-operador-K}
K:X\to X.
\end{equation}
It is not hard to show, using the compactness of the inclusion $X\hookrightarrow L^2$, for instance, 
that this operator is compact. This will be handy later on.

\medskip

Our ultimate aim is to connect problem \eqref{eq-periodic} with a `global' problem, that is,  
a problem for $(-\De)^s$ in $\R$. Thus in the reminder of this section we are going to 
obtain some preliminary regularity properties for the solution $u$ furnished by Lemma \ref{lema-exist}. 

We begin by analyzing $L^p$ regularity when $s\le \frac{1}{2}$ (when $s>\frac{1}{2}$ all functions 
in $X$ are continuous so that there is nothing to prove). It is worth remarking that the local $L^p$ theory for 
the fractional Laplacian does not seem to be perfectly understood. At the best of our knowledge, only 
results for solutions of the Dirichlet problem are available, like those in \cite{MR3393266}(see also \cite{MR3390088} and \cite{MR3341459} where the regularity theory has been done for the nonlinear operator $(-\Delta)^s_p$). But their 
proofs can be adapted to our setting to give: 

\begin{lemma}\label{lema-reg-Lp}
Assume $s \le \frac{1}{2}$ and let $f$ be a $2\pi$-periodic function with $f\in L^q(0,2\pi)$, 
where $q\ge 2/(1+2s)$. Let $u\in X$ be the unique solution of \eqref{eq-periodic} given by 
Lemma \ref{lema-exist}. Then:

\begin{itemize}

\item[(a)] If $q<\frac{1}{2s}$, then $u \in L^m(0,2\pi)$, where $m=\frac{q}{1-2qs}$. Moreover, 
\begin{equation}\label{eq-estim-2}
\| u \|_{L^m(0,2\pi)} \le C \| f \|_{L^q(0,2\pi)}.
\end{equation}

\item[(b)] If $q> \frac{1}{2s}$, then $u \in L^\infty(0,2\pi)$, and
\begin{equation}\label{eq-estim-3}
\| u \|_{L^\infty(0,2\pi)} \le C \| f \|_{L^q(0,2\pi)}.
\end{equation}
\end{itemize}
\end{lemma}

\bigskip

In order to prove Lemma \ref{lema-reg-Lp}, we need to adapt Proposition 2.2 in \cite{MR3390088} (see also Proposition 4 in \cite{MR3393266})
for our operator $\mathcal{L}$. This can be made in a straightforward way, so we will not 
include the proof. 

\begin{lemma}\label{lema-convexidad}
Let $\Phi$ be a convex, Lipschitz function. Then for every $u\in X$,
$$
\mathcal{L} \Phi(u) \le \Phi'(u) \mathcal{L} u \qquad \hbox{in } (0,2\pi)
$$
in the weak sense.
\end{lemma}

\bigskip

\begin{proof}[Proof of Lemma \ref{lema-reg-Lp}]
We first notice that it is enough to prove the Lemma with the additional requirement 
$f\ge 0$. Indeed, in the general case we can write $u=v-w$, where $v$, $w$ are the 
solutions of \eqref{eq-periodic} with $f$ replaced by $f^+$ and $f^-$. Then we can 
apply the obtained estimates for $v$ and $w$, which lead to the corresponding one 
for $u$.

Thus we will assume for the rest of the proof that $f\ge 0$, so that also $u\ge 0$ 
by Lemma \ref{lema-exist}. Let $\beta\ge 1$, $T>0$. We define for $t>0$
$$
\Phi(t)=\left\{
\begin{array}{ll}
t^\beta & 0\le t\le T\\[0.5pc]
\beta T^{\beta-1} (t-T) + T^\beta, & t>T.
\end{array}
\right.
$$
Then $\Phi$ is a convex, nondecreasing,  Lipschitz function which verifies $\Phi(t) \le t\Phi'(t)$ and
$\Phi'(t) \le \beta \Phi(t)^\frac{\beta-1}{\beta}$. By Lemma \ref{lema-convexidad}, we see that 
$$
\mathcal{L} \Phi (u) + \Phi (u) \le \Phi'(u)(f-u) + \Phi(u)\le \Phi'(u) f,
$$
in the weak sense. Testing with $\Phi(u)$ we arrive at
\begin{equation}\label{eq-aux-2}
\| \Phi(u) \|^2 \le \int_0^{2\pi} \Phi(u)\Phi'(u) f.
\end{equation}
In order to be able to use the continuity of the embedding $X \hookrightarrow L^{2^*_s}(0,2\pi)$ in 
\eqref{eq-aux-2}, we will assume throughout the rest of the proof that $s<\frac{1}{2}$. 
When $s=\frac{1}{2}$, we can use the embedding into $L^\theta (0,2\pi)$ for some arbitrary, suitably large 
exponent $\theta$ with only minor variations in the inequalities that follow. Thus
\begin{align}\label{eq-aux-3}
\nonumber \| \Phi(u) \|_{L^{2^*_s}}^2 & \le C \int_0^{2\pi} \Phi(u)\Phi'(u) f \le 
C \| f \|_{L^q} \left( \int_0^{2\pi} (\Phi(u)\Phi'(u))^{q'} \right)^\frac{1}{q'}\\ 
& \le C \beta \| f \|_{L^q} \left( \int_0^{2\pi} \Phi(u)^{\frac{(2\beta-1)}{\beta}q'} \right)^\frac{1}{q'}.
\end{align}
We now proceed differently depending on whether $q$ is less than or greater than $\frac{1}{2s}$, 
that is, we make different choices of $\beta$ in parts (a) and (b).

\medskip

To prove part (a), we just set $\beta=\frac{m}{2^*_s}$. A little algebra shows that 
$\frac{2\beta-1}{\beta} q'=2^*_s$. Thus from \eqref{eq-aux-3}
$$
\| \Phi(u) \| _{L^{2^*_s}}^\frac{2^*_s}{m}  \le C \| f\|_{L^q}.
$$
We can let $T\to +\infty$ in this inequality to obtain that
$$
\| u^\beta \|_{L^{2^*_s}}^\frac{2^*_s}{m}  \le C \| f\|_{L^q},
$$
which is precisely estimate \eqref{eq-estim-2}, since $2^*_s \beta=m$, and part (a) is proved.

\medskip

As for part (b), we will use Moser's iteration argument. It is easily seen that the arguments 
at the beginning of the proof can be adapted to give a similar inequality as \eqref{eq-aux-3} 
but with $u$ replaced by $v=u+1$ and $f$ replaced by $f+1$. If we now assume that $v \in L^{2\beta q'}$ 
for some $\beta \ge 1$, we can let $T\to +\infty$ in the modified version of 
\eqref{eq-aux-3} to get
$$
\| v^\beta \|_{L^{2^*_s}}^2 \le C\beta \| f +1\|_{L^q} \| v^{2\beta-1}\| _{L^{q'}}
\le C\beta \| f + 1 \|_{L^q} \| v^{2\beta}\| _{L^q},
$$
since $v\ge 1$. This inequality can be rewritten as
\begin{equation}\label{eq-aux-9}
\| v \|_{L^{2^*_s\beta}} \le (C\beta \| f +1 \|_{L^q})^\frac{1}{2\beta} \| v \| _{L^{2\beta q'}}.
\end{equation}
We now particularize \eqref{eq-aux-9} by setting $\beta_j=\chi^j$, $j=0,1,\ldots$, where 
$\chi=\frac{2^*_s}{2q'}$. It is not hard to check that $q>\frac{1}{2s}$ implies $\chi>1$. 
Moreover, since $2q'<2^*_s$, we see that $v\in L^{2\beta_0q'}$. Thus for every nonnegative 
integer $j$, we have
$$
\| v \|_{L^{\beta_{j+1}r}} \le (C\beta_j \| f +1 \|_{L^q})^\frac{1}{2\beta_j} \| v \| _{L^{\beta_j r}}.
$$
where $r=2q'$.

Iterating this inequality we see that
$$
\| v \|_{L^{\beta_{j+1}r}} \le \left( C\| f +1 \|_{L^q}\right) ^{\frac{1}{2} \sum_{j=0}^k \frac{1}{\chi^j}} 
\chi ^{  \frac{1}{2} \sum_{j=0}^k \frac{j}{\chi^j}}\| v \| _{L^r}.
$$
Since $\chi>1$ we see that the sums in the right-hand side give rise to convergent series. Thus we 
are allowed to let $j\to +\infty$ to obtain
\begin{equation}\label{eq-aux-4}
\| u+1 \|_{L^\infty} \le C \|f+1\|_{L^q}^D \| u+1 \| _{L^{2q'}}
\end{equation}
for some positive constants $C,D$. Since the norm of $u$ in $L^{2q'}$ can be controlled in terms 
of the norm in $X$, we immediately see from \eqref{eq-aux-4} that, for a possibly different constant 
$C$:
\begin{equation}\label{eq-aux-5}
\| u \|_{L^\infty} \le C \left[(\|f\|_{L^q}^D +1)(\| u  \|+1)+1\right].
\end{equation}
We finally obtain \eqref{eq-estim-3}. For his sake, simply notice that on one hand 
$w=u/\|f\|_{L^q}$ solves \eqref{eq-periodic} with $f$ replaced by $g=f/\|f\|_{L^q}$. 
Since $g$ has $L^q$-norm equal to one, \eqref{eq-aux-5} gives
$$
\| v\|_{L^\infty} \le C( \| v\| + 1)
$$
for a different constant $C$. That is,
$$
\| u\|_{L^\infty} \le C( \| u\| + \|f \|_{L^q}).
$$
On the other hand, by Lemma \ref{lema-exist}, we can bound the norm of $u$ in terms of the norm 
$\| f\|_{X'}$. Since $\| f\|_{X'}\le \| f\|_{L^q}$, we see that \eqref{eq-estim-3} holds, 
concluding the proof of the lemma.
\end{proof}

\bigskip

\begin{remark}{\rm When $s<\frac{1}{2}$ and $q$ equals the threshold value $\frac{1}{2s}$, it can 
be seen as in Theorem 15 of \cite{MR3393266} that $u\in L^m$ for every $m>1$. 
}\end{remark}

\bigskip

\section{A global problem. Further regularity}\label{sect-global}
\setcounter{equation}{0}

In this section we will introduce a global periodic problem which is uniquely solvable 
and taht will allow us to prove Theorem \ref{th-equivalencia}. Consider:
\begin{equation}\label{eq-global}
(-\De)^s u + u = f \quad \hbox{in } \R,
\end{equation}
where $f$ is a $2\pi$-periodic, bounded function, and we are only interested in 
bounded solutions.

We first solve problem \eqref{eq-global} when $f$ is smooth enough, which will permit us 
restrict our attention to classical solutions.

\begin{lemma}\label{lema-exist-global}
Let $f\in C^\al(\R)$ be $2\pi$-periodic, where $\al\in (0,1)$. Then problem 
\eqref{eq-global} admits a unique bounded classical solution $u \in C^{2s+\al}(\R)$, 
which is also $2\pi$-periodic. Moreover, 
\begin{equation}\label{eq-estim-infty}
\| u\|_{L^\infty (\R)} \le \| f\|_{L^\infty(\R)}.
\end{equation}
\end{lemma}

\begin{proof}
We begin by proving the existence of a solution $u$ verifying \eqref{eq-estim-infty}. 
Choose $M>0$ and consider the problem
\begin{equation}\label{eq-problema-finito}
\left\{
\begin{array}{ll}
(-\De)^s u + u = f & \hbox{in } (-M,M),\\
\ \ u=0 & \hbox{in } \R\setminus (-M,M).
\end{array}
\right.
\end{equation}
It is clear that $\underline{u}=-\|f\|_{L^\infty(\R)}$, $\overline{u}=\|f\|_{L^\infty(\R)}$ are a 
pair of order sub and supersolutions of \eqref{eq-problema-finito}. Thus there exists a viscosity 
solution $u_M$ of \eqref{eq-problema-finito} verifying 
$$
\| u_M \|_{L^\infty(\R)} \le \|f \|_{L^\infty(\R)},
$$
(see for instance Theorem A.1 in the Appendix of \cite{2016arXiv170402597}).
Now we can use the H\"older estimate given by Theorem 12.1 of \cite{MR2494809} 
to deduce the existence of $\al\in (0,1)$ such that, for every $R>0$, there exist  
$M_0$ and $C$ such that 
$$
\| u_M \|_{C^\al[-R,R]}\le C,
$$
for $M\ge M_0$, where $C$ does not depend on $M$. Thus we may choose a sequence 
$M_n\to +\infty$ such that $u_{M_n}\to u$ uniformly on compact sets with $u\in C(\R)$. 
Using Corollary 4.7 in \cite{MR2494809} we see that $u$ is a viscosity solution of 
\eqref{eq-global}. By standard regularity, $u\in C^{2s+\al}(\R)$. Check for instance 
Proposition 2.8 in \cite{MR2270163}.

To prove uniqueness, it suffices to show that the associated homogeneous problem 
admits only the trivial solution. Thus assume $v\in C^{2s+\al}(\R)\cap L^\infty(\R)$ 
verifies $(-\De)^s v +v =0$ in $\R$. Choose a sequence $\{x_n\}$ such that 
$v(x_n)\to \sup_{\R} v$. 

Consider the functions $v_n(x)=v(x+x_n)$, which are a uniformly bounded sequence of 
solutions of the same equation. Arguing similarly as above we obtain, by passing to a 
subsequence, that $v_n\to \bar v$ locally uniformly, where $\bar v$ is a viscosity -- hence 
classical -- solution of $(-\De)^s \bar v+ \bar v=0$ in $\R$ with $\bar v(0)= 
\sup_{\R} v=\sup_{\R} \bar v$. Now, since $\bar v$ attains its global maximum at $x=0$, we can 
evaluate the equation at this point to obtain
$$
\sup_{\R} v\le 0.
$$
This shows that $v\le 0$ and a similar argument gives $v \equiv 0$ in $\R$, hence the uniqueness.

Finally, the periodicity of $u$ is a consequence of uniqueness and the periodicity of 
$f$. Indeed, the function $u(\cdot +2\pi)$ is a solution of \eqref{eq-global} and uniqueness 
implies $u(\cdot +2\pi)\equiv u$. The proof is concluded.
\end{proof}

\bigskip

The next result is our first connection between the periodic problem 
\eqref{eq-periodic} and the global one \eqref{eq-global}.

\begin{lemma}\label{lema-reg-c2-alfa}
Assume $f\in C^\al(\R)$ is $2\pi$-periodic and let $u\in X$ be the unique 
weak solution of \eqref{eq-periodic}. Then $u\in C^{2s+\al}(\R)$ and $u$ is 
the unique solution of \eqref{eq-global}.
\end{lemma}

\begin{proof}
Let $v\in C^{2s+\al}(\R)$ be the unique solution of \eqref{eq-global}. The lemma 
will be proved if we show that $v$ is a weak solution of \eqref{eq-periodic}. 
To show this, observe that by Lemma \ref{lema-operador} we have
\begin{equation}\label{autonom}
\int_0^{2\pi} (v(x)-v(y)) H(x-y) dy = f(x) - v(x) \quad \hbox{for every } x\in (0,2\pi),
\end{equation}
where the integral is to be understood in the principal value sense. 
Take an arbitrary $2\pi$-periodic function $\phi\in C^1(\R)$. By 
the regularity of both $v$ and $\phi$ the integral 
$$
I:= \frac{1}{2} \int_0^{2\pi} \hspace{-2mm} \int_0^{2\pi} (v(x)-v(y)) (\phi(x)-\phi(y)) H(x-y) dy dx
$$
is absolutely convergent. Our intention is to apply Fubini's theorem to this integral, but 
some care is needed since in the procedure some of the involved integrals turn out to be 
defined only in the principal value sense.

Thus we choose $0<\delta<\e$ and let 
\begin{align*}
Q_{\e,\delta}:= &\Big[([0,\e]\times [\e,2\pi-\e] )\cup ([\e,2\pi-\e]\times [0,2\pi])\\ 
& \cup ([2\pi-\e,2\pi]\times [\e,2\pi-\e]) \Big] \cap \{|y-x|>\delta\}.
\end{align*}
The part of the integral $I$ taken in $Q_{\e,\delta}$ can be manipulated since everything 
is smooth there. Thus splitting the integral in two parts and interchanging the 
variables in the standard way (taking into account that $H$ is symmetric) we obtain:
\begin{align*}
I & =\frac{1}{2} \int \hspace{-2mm} \int _{Q_{\e,\delta}} (v(x)-u(y)) (\phi(x)-\phi(y)) H(x-y) dy dx
+o(1)\\
& =  \int \hspace{-2mm} \int _{Q_{\e,\delta}} \phi(x) (v(x)-v(y)) H(x-y) dy dx + o(1) =:J_{\e,\delta}+o(1).
\end{align*}
Here $o(1)$ is a number which goes to zero as $\e$ and $\de$ go to zero. By Fubini's theorem, the 
last integral can be written in terms of iterated integrals as follows:
\begin{align*}
J_{\e,\delta} & = \left( \int_0^{\e-\delta}\hspace{-2mm} \int_\e^{2\pi-\e} +\int_{\e-\delta}^\e \int_{x+\delta}^{2\pi-\e} 
 + \int_\e^{2\pi-\e} \hspace{-2mm}\int_0^{x-\delta} +\int_\e^{2\pi-\e}\hspace{-1mm} \int_{x+\delta} ^{2\pi}\right.\\
& + \left.\int_{2\pi-\e}^{2\pi-\e+\delta} \hspace{-1mm}\int_\e^{x-\delta}+ 
\int_{2\pi-\e+\delta}^{2\pi} \int_\e^{2\pi-\e}\right) \phi(x) (v(x)-v(y)) H(x-y)dy dx.
\end{align*}
Arguing as in Lemma 3.2 in \cite{2016arXiv170402597} we may let $\de\to 0$ to obtain 
that $J_{\e,\de}\to J_\e$, given by
\begin{align*}
J_\e: & = \left( \int_0^{\e} \int_\e^{2\pi-\e}  + \int_\e^{2\pi-\e}\hspace{-2mm} \int_0 ^{2\pi} 
+ \int_{2\pi-\e}^{2\pi} \int_\e^{2\pi-\e}\right) \phi(x) (v(x)-v(y)) H(x-y)dy dx\\
& = \left( \int_0^{\e} \int_\e^{2\pi-\e}  + \int_{2\pi-\e}^{2\pi} \int_\e^{2\pi-\e}\right) 
\phi(x) (v(x)-v(y)) H(x-y)dy dx\\
& + \int_\e^{2\pi-\e}  \phi(x) (f(x)-v(x))dx,
\end{align*}
where in the last equatity we have used \eqref{autonom}. It is not hard to see that the first two integrals in the last right-hand side converge to zero as $\e\to 0$. 
Indeed, we can write for the first one 
\begin{align*}
\int_0^{\e} & \int_\e^{2\pi-\e}\phi(x) (v(x)-v(y)) H(x-y)dy dx \\
& = \left(\int_0^{\e} \int_\e^{\pi} + \int_0^{\e} \int_\pi^{2\pi-\e}\right) \phi(x) (v(x)-v(y)) H(x-y)dy dx. 
\end{align*}
We estimate only the first of these integrals, the other one being taken care with in a similar 
fashion. By the regularity of $u$ and $\phi$ we can find $C>0$ so that
\begin{align*}
\int_0^{\e} & \int_\e^{\pi} |\phi(x)| |v(x)-v(y)| H(x-y)dy dx \le C \int_0^{\e} \int_\e^{\pi} 
|x-y|^{{-2s}}dy dx\\
& = \frac{C}{{1-2s}} \int_0^\e ((\pi-x)^{{1-2s}}-(\e-x)^{{1-2s}}) dx \\[.75pc]
& = O(\e) + O(\e^{{2-2s}}) = o(1).
\end{align*}
The remaining integral in the expression of $J_\e$ is dealt with in a similar manner. 
Thus we deduce by letting $\e\to 0$ that
\begin{align*}
\frac{1}{2} \int_0^{2\pi}& \hspace{-2mm} \int_0^{2 \pi} (v(x)-v(y)) (\phi(x)-\phi(y)) H(x-y)dydx\\
& + \int_0^{2\pi} v(x)\phi(x)dx = \int_0^{2\pi} f(x)\phi(x) dx.
\end{align*}
Since $\phi\in C^1(\R)$ is arbitrary this equality also holds for every $\phi\in X$ by density 
and this shows that $v$ is a weak solution to \eqref{eq-periodic}. 
Thus $v=u$ and the proof is concluded.
\end{proof}

\bigskip

The last piece in our puzzle concerns $C^\al$ regularity for the solution of 
\eqref{eq-global} when the right-hand side is bounded. It is an essential ingredient 
in order that bootstrapping works properly. 

\begin{lemma}\label{lema-reg-c-alfa}
Assume $f\in L^\infty(\R)$ is $2\pi$-periodic and let $u\in X$ be the 
unique solution of \eqref{eq-periodic}. Then $u \in C^\al(\R)$ for every 
$\al \in (0,2s)$ and 
$$
\| u\|_{C^\al(\R)} \le C \| f\|_{L^\infty(\R)}.
$$
\end{lemma}

\begin{proof}
Let $\{\rho_n\}_{n=1}^\infty$ be a sequence of regularizing kernels in the usual sense. 
Denote
$$
u_n(x)=\int u(x-z) \rho_n(z)dz, \quad f_n(x)=\int f(x-z) \rho_n(z)dz, 
$$
where the integrals are taken in $\R$. It is well-known that $u_n,f_n\in C^\infty(\R)$, 
 with the bounds $\| u_n\|_{L^\infty(\R)} \le \| u\|_{L^\infty(\R)}$, 
$\| f_n\|_{L^\infty(\R)} \le \| f\|_{L^\infty(\R)}$. Moreover, $f_n\to f$ in $L^q(0,2\pi)$ 
for every $q>1$. 

We claim that $u_n$ is a weak solution of 
$$
\mathcal{L} u_n + u_n =f_n \quad \hbox{in }(0,2\pi).
$$
To check this, take a $2\pi$-periodic function $\phi \in C^1(\R)$ and use as a test 
function in \eqref{eq-periodic} the shifted function $\phi(\cdot+z)$, where $z\in \R$. 
After changing variables in the resulting expression, we see that
\begin{align*}
\int_z^{2\pi+z} & \hspace{-2mm} \int_z^{2\pi+z} (u(x-z)-u(y-z))(\phi(x)-\phi(y)) H(x-y) dy dx \\
& = \int_z^{2\pi+z} (f(x-z)-u(x-z))\phi(x) dx.
\end{align*}
By periodicity, all the integrals can be taken in any interval of length $2\pi$, in particular 
in $[0,2\pi]$. Thus if we multiply by $\rho_n(z)$ and integrate with respect to $z$ in $\R$ we see that
\begin{align*}
\int \rho_n(z) & \left( \int_0^{2\pi} \hspace{-2mm} \int_0^{2\pi} (u(x-z)-u(y-z))(\phi(x)-\phi(y)) H(x-y) 
dy dx\right) dz \\
& = \int \rho_n(z) \left( \int_0^{2\pi} (f(x-z)-u(x-z))\phi(x) dx \right) dz.
\end{align*}
We can use Fubini's theorem with no worries since all integrals are absolutely convergent. Thus
$$
\int_0^{2\pi} \hspace{-2mm} \int_0^{2\pi} (u_n(x)-u_n(y))(\phi(x)-\phi(y)) H(x-y)dy dx
=\int_0^{2\pi} (f_n(x)-u_n(x))\phi(x) dx.
$$
Hence the claim follows.

Now, owing to Lemma \ref{lema-reg-c2-alfa}, we see that $u_n$ is a classical solution of the global 
problem
$$
(-\De)^s u_n + u_n = f_n \quad \hbox{in } \R,
$$
and we are in a position to apply the standard regularity for the fractional Laplacian. 
In particular, by Proposition 2.9 in \cite{MR2270163} we see that the $C^\al$ norm of 
$u_n$ can be estimated in terms of its $L^\infty$ norm and that of $f$, for every $\al \in (0,2s)$. 
Thus
\begin{equation}\label{eq-aux-6}
\|u_n \|_{C^\al(\R)} \le C(\|u_n\|_ {L^\infty(\R)} + \|f_n\|_ {L^\infty(\R)})\le 
C(\|u\|_ {L^\infty(\R)} + \|f\|_ {L^\infty(\R)}).
\end{equation}
Finally, we see by using Lemma \ref{lema-reg-Lp}, part (b) that
$$
\| u_n-u_m\| _{L^\infty(\R)} \le C \|f_n -f_m \|_{L^q(0,2\pi)}
$$
for every $q>\frac{1}{2s}$. Therefore $u_n\to u$ uniformly in $\R$, and we deduce that 
estimate \eqref{eq-aux-6} is also true for $u$ by passing to the limit. The proof is concluded.
\end{proof}

\bigskip

We are finally ready for proving of our first main result.

\medskip

\begin{proof}[Proof of Theorem \ref{th-equivalencia}]
The proof is more or less standard, based on an iteration and bootstrapping. The important tools 
in the procedure will be given by Lemmas \ref{lema-reg-Lp}, \ref{lema-reg-c2-alfa} and \ref{lema-reg-c-alfa}.

We will show first that $u\in L^\infty$. Notice that because of the inclusion of our space 
$X$ in $H^s(0,2\pi)$ we always have $u\in L ^\infty$ if $s>\frac{1}{2}$. Thus we may assume 
in this step that $s< \frac{1}{2}$, the case $s=\frac{1}{2}$ being handled similarly.

Since $u\in X$, we deduce $u\in L^{q_0}$, where $q_0=2^*_s$. On the other hand, we have 
$\mathcal{L}u+u =g$, where 
$$
g(x)=u(x)+f(u(x)),
$$ 
and by our hypothesis \eqref{eq-hipo} on $f$, we deduce $g\in L^\frac{q_0}{p}$. Observe that 
$$
\frac{q_0}{p}> \frac{2^*_s}{2^*_s-1}=\frac{2}{1+2s},
$$
so we are in a position to apply Lemma \ref{lema-reg-Lp}. 

If $\frac{q_0}{p}>\frac{1}{2s}$ we deduce that $u\in L^\infty$. Thus assume 
for the moment $\frac{q_0}{p}<\frac{1}{2s}$. Then $u\in L^{q_1}$, where 
$$
q_1=\frac{q_0}{p-2q_0 s}.
$$
It is not hard to see that $p<2^*_s-1$ implies $q_1>q_0$. This procedure can be iterated to obtain that, 
whenever  $\frac{q_j}{p}<\frac{1}{2s}$, then $u\in L^{q_{j+1}}$, where 
\begin{equation}\label{eq-def-qj}
q_{j+1}=\frac{q_j}{p-2q_j s},
\end{equation}
while $u\in L^\infty$ if  $\frac{q_j}{p}>\frac{1}{2s}$. Moreover, the sequence $\{q_j\}$ is 
increasing. 

We claim that $\frac{q_j}{p} \ge \frac{1}{2s}$ for some $j$. Otherwise, we would have that the 
sequence $\{q_j\}$ is bounded, which would imply $q_j\to q$. Passing to the limit in \eqref{eq-def-qj}, 
we see that 
$$
q=\frac{p-1}{2s}<\frac{2^*_s-2}{2s}=2^*_s,
$$
which is a contradiction with $q>q_0=2^*_s$.

Thus $\frac{q_j}{p}\ge \frac{1}{2s}$ for some $j$. If the inequality is strict, we have 
$u\in L^\infty$. Otherwise, we can use that $u\in L^{q_{j-1}+\de}$ for some small $\de$ 
which implies 
$$
u\in L^{\frac{q_{j-1}+\de}{p-2q_{j-1}s}}.
$$
However, this last exponent is greater than $\frac{p}{2s}$ for small $\de$ and we conclude as 
before.

To summarize, we have shown that $u \in L^\infty$. But then $g\in L^\infty$ and we can 
use Lemma \ref{lema-reg-c-alfa} to deduce that $u\in C^\al$ for every $\al \in (0,2s)$. 
Using the Lipschitz condition on $f$, this entails $g\in C^\al$ for every $\al\in (0,1)$. 
Thus Lemma \ref{lema-reg-c2-alfa} gives $u\in C^{2s+\al}$ for every $\al\in (0,1)$, and 
$u$ is a $2\pi$-periodic classical solution of \eqref{eq-semilineal-global}, as was to be proved.
\end{proof}

\bigskip

\begin{remark}\label{remark-control-normas} {\rm An inspection of the proof of Theorem \ref{th-equivalencia} reveals 
that the $C^{2s+\al}$ norms of the solution $u$ can be bounded in terms of the norm in $X$. 
In particular, when $f$ verifies $|f(t)|\le C |t|$ for some $C>0$ and all $t\in \R$, then 
$$
\| u\|_{C^{2s+\al}} \le C \| u \|.
$$
}\end{remark}

\bigskip

\section{An eigenvalue problem}\label{sect-eigen}
\setcounter{equation}{0}

In this section we will analyze the eigenvalue problem for the operator $\mathcal{L}$, 
which will be useful in the context of bifurcation theory. The eigenvalue problem will be:
\begin{equation}\label{eq-eigen-periodic}
\mathcal{L} u = \la u \quad \hbox{in } (0,2\pi).
\end{equation}
As a consequence of Theorem \ref{th-equivalencia}, weak solutions of \eqref{eq-eigen-periodic} 
correspond to classical $2\pi$-periodic solutions of 
\begin{equation}
(-\De)^s u  = \la u \quad \hbox{in } \R.
\tag{\ref{eq-eigen-global}}
\end{equation}
It is worth remarking that this eigenvalue problem has been implicitly considered in 
\cite{JCFNthesis}, where it was shown that $\la_k=k^{2s}$ are eigenvalues for every 
integer $k\ge 1$ (see Proposition 2.3.4 there). The corresponding eigenfunctions are 
$\cos(kx)$ and $\sin(kx)$. We will show indeed that, aside the trivial eigenvalue $\la_0=0$ 
with a constant eigenfunction, these are the only ones.

\begin{lemma}\label{lema-eigen}
The only eigenvalues of problem \eqref{eq-eigen-periodic} are $\la_0=0$, 
with a constant eigenfunction and $\la_k = k^{2s}$, $k=1,2,\ldots$, with 
an eigenspace spanned by $\cos(kx)$ and $\sin(kx)$. 
\end{lemma}

\begin{proof}
Problem \eqref{eq-eigen-periodic} is clearly equivalent to $u=(\la+1)Ku$, where 
$K$ is the solution operator given in \eqref{eq-operador-K}. This shows that 
$\la\ne -1$ and 
$$
Ku =\frac{1}{\la+1} u.
$$
The compactness of $K$ implies that every eigenvalue is of finite multiplicity. 

Next observe that a constant eigenfunction $u$ can only be associated to the eigenvalue $\la_0=0$.
Conversely, periodic eigenfunctions associated to $\la_0$ are necessarily constant, since they are 
bounded $s-$harmonic functions (cf. \cite{MR3477075}). Hence we may restrict in what follows to 
nonconstant eigenfunctions. 

Thus let $u$ be a $2\pi$-periodic nonconstant eigenfunction associated to an eigenvalue $\la$. By bootstrapping and 
regularity theory for the fractional Laplacian (cf. \cite{MR2270163}), we have $u\in C^\infty(\R)$. 
The same regularity, together  with the periodicity of $u$ implies that $u'$ is bounded. Thus $u'$ is 
also an eigenfunction associated to $\la$. Proceeding inductively, we see that derivatives of 
$u$ up to all orders are also eigenfunctions associated to $\la$. Since the eigenspace associated to $\la$ is 
finite dimensional, there exists an integer $m$ and real numbers $c_0,c_1,\ldots,c_{m-1}$, not all 
zero, such that 
$$
u^{(m)}+c_{m-1} u^{(m-1)}+\cdots + c_1 u'+c_0 u=0 \qquad \hbox{in }\R.
$$
This means that $u$ is a solution of a homogeneous, linear ordinary differential equation with 
constant coefficients. It is well-known that the only periodic nonconstant solutions of this equation 
are of the form
$$
u= \sum_{i=1}^n a_i \cos(\mu_i x) + b_i \sin (\mu_i x),
$$
where $2n\le m$, $a_i,b_i\in \R$ and the numbers $\mu_i$ can be chosen to be positive. 
Substituting this expression into \eqref{eq-eigen-global}, we see that 
$$
\sum_{i=1}^n a_i \mu_i^{2s} \cos(\mu_i x) + b_i \mu_i^{2s} \sin (\mu_i x)=
\la \sum_{i=1}^n a_i \cos(\mu_i x) + b_i \sin (\mu_i x),
$$
which readily implies $n=1$ and $\mu_1^{2s}=\la$. The $2\pi$-periodicity of $u$ 
then yields $\mu_1= k$ for a nonnegative integer $k$, so that $\la=k^{2s}$ and 
$u$ is a linear combination of $\cos(kx)$ and $\sin(kx)$, as we wanted to show. The 
proof is concluded.
\end{proof}

\begin{remark}{\rm
A {stronger statement than that of Lemma \ref{lema-eigen}} is actually available by 
using the Fourier transform. Namely, the only solutions $u$ of \eqref{eq-eigen-global} 
verifying
\begin{equation}\label{eq-cond-integrab1}
\int_{-\infty}^\infty \frac{|u(x)|}{(1+|x|)^{1+2s}} dx<+\infty
\end{equation}
are $\cos(\la^\frac{1}{2s} x)$ and $\sin(\la^\frac{1}{2s}x)$, provided that $\la>0$. 
Observe in passing that condition \eqref{eq-cond-integrab1} is necessary in order that 
$u$ be a solution of \eqref{eq-eigen-global} in whatever sense, and that no boundedness 
nor periodicity assumptions are placed on $u$. We sketch a proof next.

Condition \eqref{eq-cond-integrab1} implies that $u$ defines a tempered distribution, so 
that its Fourier transform is well-defined. After transforming equation \eqref{eq-eigen-global}, 
we arrive at 
\begin{equation}\label{eq-aux-8}
|\xi|^{2s} \hat u =\la \hat u
\end{equation}
in the sense of distributions. This implies that $\la\ge 0$ and that the support of 
$\hat u$ is contained in $\{-\la^\frac{1}{2s},\la^\frac{1}{2s}\}$. Arguing as in 
Section 3 of \cite{MR3348929}, we can show that actually \eqref{eq-aux-8} is equivalent to 
$|\xi|^2 \hat u=\la^{1/s} \hat u$, which in turn is equivalent to $-u''=\la^{1/s} u$ in the 
sense of distributions. Hence the statement.
}\end{remark}

\bigskip

Using Lemma \ref{lema-eigen} and the usual orthogonality relations between sines 
and cosines, it is not difficult to deduce that 
\begin{equation}\label{eq-ortogonalidad}
\begin{array}{c}
\left< \cos(ix),\sin(jx)\right> = 0, \quad i,j\in \mathbb{N}\\[.25pc]
\left< \cos(ix),\cos(jx)\right> = (i^{2s}+1) \de_{ij} \pi, \quad (i,j)\ne (0,0)\\[.25pc]
\left< \sin(ix),\sin(jx)\right> = (i^{2s}+1) \de_{ij} \pi, \quad (i,j)\ne (0,0),
\end{array}
\end{equation}
where $\de_{ij}$ is Kronecker's delta and $\langle\cdot,\cdot \rangle$ is the inner product defined in \eqref{producto_interior}. Thus the following variational 
characterization, which will be handy in Section \ref{sect-variacional} below, is to be 
expected.

\begin{lemma}\label{lema-eigen-variacional}
Let {$k\in\mathbb{N}$} and denote
$$
F_k=\left\{ u\in X: \int_0^{2\pi} \hspace{-2mm} u(x) \cos(jx)dx = \int_0^{2\pi} \hspace{-2mm} 
u(x) \sin(jx)dx = 0, \ j={0},\ldots,k-1\right\}.
$$
Then 
\begin{equation}\label{eq-caract-variacional}
\inf_{u\in F_k} \frac{\|u\|^2}{\int_0^{2\pi} u^2} = k^{2s}+1.
\end{equation}
\end{lemma}

\begin{proof}
Denote 
$$
\mathcal{I}(u)=  \frac{\|u\|^2}{\int_0^{2\pi} u^2}
$$
and let $\mu_k=\inf_{u\in F_k} \mathcal{I}(u)$. It is easily seen that 
this infimum has to be achieved at some $u \in {F_k}$. 

Indeed, if $\{u_n\}_{n=1}^\infty$ is a 
minimizing sequence, which can be assumed by normalizing to have $L^2$ norm equal to one, 
then $\|u_n\|$ is bounded. It follows after passing to a subsequence that $u_n\rightharpoonup u$ 
weakly in $X$ for some function $u\in {F_k}$ with $L^2$ norm equal to 1. But then by the usual 
properties of weak convergence:
$$
\| u\|^2 \le \liminf_{n\to +\infty} \| u_n \|^2=\mu_k.
$$
Thus $\mathcal{I}(u)=\mu_k$, and the infimum is attained. Therefore $u$ is a critical point 
of the functional $\mathcal{I}$, so that 
\begin{equation}\label{yours}
\left< u, \varphi\right> = \mu_k \int_0^{2\pi} u \varphi, \quad \hbox{for every } \varphi \in F_k.
\end{equation}
This equality is indeed valid for every $\varphi \in X$ in the standard way. To 
see it, denote $\phi_j(x)=\cos(jx)$, $\bar\phi_j(x)=\sin(jx)$ and write 
$$
\varphi=\frac{a_0}{2} + \sum_{j=1}^{k-1} \left(a_j \phi_j+b_j \bar\phi_j\right)+ \varphi_2
$$ 
where $\varphi_2\in F_k$. Since
$$
0=\left< u,a_j \phi_j+b_j\bar\phi_j \right> =(j^{2s}+1)\int_0^{2\pi} u (a_j \phi_j+b_j\bar\phi_j )
$$
for $j=0,\ldots,k-1$, by \eqref{yours} it follows that
$$
\left<u ,\varphi\right> =\left<u ,\varphi_2\right> = \mu_k \int_0^{2\pi} u \varphi_2 =
\mu_k \int_0^{2\pi} u \varphi.
$$
This means that $\mu_k-1$ is an eigenvalue of problem \eqref{eq-eigen-periodic} with $u$ as an 
eigenfunction. By Lemma \ref{lema-eigen} we deduce that $\mu_k=j^{2s}+1$ for some $j\ge 0$ 
and that $u$ is a linear combination of $\cos(jx)$ and $\sin(jx)$. Since $u \in F_k$, we deduce 
$j\ge k$ and so the fact that $\mu_k$ is an infimum gives $j=k$, that is $\mu_k=k^{2s}+1$, as 
was to be shown. This proves \eqref{eq-caract-variacional}.
\end{proof}

\bigskip

%
%
%

\bigskip

\section{Existence of periodic solutions: bifurcation}\label{sect-bifur}
\setcounter{equation}{0}

Our main interest in this section is to prove Theorem \ref{th-bifurcacion}. It will 
be a consequence of a theorem independently due to B\"ohme \cite{MR0312348} and Marino 
\cite{MR0348570} (see also Theorem 11.4 in \cite{MR845785}). Therefore we will assume 
throughout that $f$ verifies the hypotheses in its statement. 

First of all let $\tilde f$ be an arbitrary truncation of $f$ outside 
the interval $[-1,1]$ which grows at most linearly. Since $f'(0)=0$, we can find a positive constant $C$ such that 
$|\widetilde f(t)|\le C|t|$. 
Let us see now that finding classical solutions of problem \eqref{eq-bifur} with small 
$L^\infty$ norm is equivalent to finding weak solutions of 
\begin{equation}\label{eq-rabinowitz}
\mu( \mathcal{L} v +v) = v+\tilde f(v), \quad v\in X,
\end{equation}
with small $\| \cdot \|$ norm. 

\begin{lemma}\label{lema-truncamiento}
For every $\mu_0>0$, there exists $C>0$ such that, if $v\in X$ is a weak solution of \eqref{eq-rabinowitz} with 
$\| v \|\le \e$ and some $\mu \ge \mu_0$, then there exists a classical solution $u$ of \eqref{eq-bifur} 
with $\la = 1-\mu$ and $f$ replaced by $\tilde f$, which is periodic with period $2\pi \mu^{-\frac{1}{2s}}$ 
and verifies $\| u \|_{L^\infty(\R)} \le C\e$. Thus if $\e$ is small enough, $u$ is also a solution 
of problem \eqref{eq-bifur}.
\end{lemma}

\begin{proof}
Observe that by Remark \ref{remark-control-normas}, there exists $C>0$ such that $\| v\|_{L^\infty(0,2\pi)} 
\le C\|v\| \le C\e$, where $C$ depends on a lower bound for $\mu$, i. e. on $\mu_0$.

On the other hand, since $\tilde f$ is subcritical, 
by Theorem \ref{th-equivalencia}, problem \eqref{eq-rabinowitz} is equivalent to 
$$
(-\De)^s v= \frac{1-\mu}{\mu} v +\frac{1}{\mu} \tilde f(v) \quad \hbox{in } \R.
$$
Letting $u(x)= v(\mu^\frac{1}{2s} x)$ and $\la=1-\mu$, we see that $u$ is a periodic 
solution of \eqref{eq-bifur} (with $f$ replaced by $\tilde f$) with period $2\pi \mu^{-\frac{1}{2s}}$, and 
$\| u\|_{L^\infty(\R)}= \| v\|_{L^\infty(0,2\pi)} \le C\e$. Moreover, when $\e$ 
is small enough, we also have $\| u\|_{L^\infty(\R)} \le 1$, hence $u$ is a solution of 
the original problem \eqref{eq-bifur}. The proof is concluded.
\end{proof}

\bigskip

\begin{proof}[Proof of Theorem \ref{th-bifurcacion}]
Let $F(t)=\int_0^t \tilde f(\tau)d\tau$, and for $u\in X$, consider the functional 
$$
I(u)= \frac{1}{2} \int_0^{2\pi} u^2 + \int_0^{2\pi} F(u).
$$
Observe that $I$ is well-defined on $X$ by the subcriticality of $\tilde f$. 
We introduce the operators $L:X\to X$ and $G:X\to X$ by
\begin{align*}
& \left< Lu,\varphi\right> = \int_0^{2\pi} u\varphi \\
& G(u) \varphi = \int_0^{2\pi} \tilde f(u) \varphi.
\end{align*}
Next, denote by $D$ the duality mapping between $X'$ and $X$. That is, for 
$h\in X'$, let $Dh$ be the unique $z\in X$ such that 
$$
\left< z, u\right> = h(u), \quad \hbox{for every } u \in X.
$$
Then it can be easily shown that 
$$
DI'(u)=Lu +G(u).
$$
Moreover, equation \eqref{eq-rabinowitz} is equivalent to $Lu+G(u)=\mu u$. In order to be able 
to apply Theorem 11.4 in \cite{MR845785} we need to have:

\begin{itemize}

\item[(a)] An isolated eigenvalue $\mu_0$ of $L$ with finite multiplicity.

\medskip
\item[(b)] $G(u)=o(\|u\|)$ as $\| u \|\to 0$.
\end{itemize}

As for (a), notice that $\mu_0$ is an eigenvalue of $L$ in $X$ if and only if 
$\frac{1}{\mu_0}-1$ is an eigenvalue of \eqref{eq-eigen-periodic}, that is, by Lemma \ref{lema-eigen},
$\mu_0=(1+k^{2s})^{-1}$ for some $k\ge 0$. Moreover, for $k\ge 1$, $\mu_0$ has 
multiplicity two.

To show (b), choose $q>1$ such that 
$$
\cfrac{2}{1+2s} \le q < \frac{2}{1-2s}=2^*_s,
$$
and let $\theta>1$ be close to 1, so that $q\theta$ verifies the same inequalities. 
Choose $\e>0$. Then there exists $\de>0$ such that $|\tilde f(t)|\le \e|t|$ if $|t|\le \de$, 
while $|\tilde f(t)| \le C(\delta)|t|^\theta$ if $|t|\ge \de$. Then 
\begin{align*}
\int_0^{2\pi} |\tilde f(u)|^q & = \int_{|u|\le \de}  |\tilde f(u)|^q +\int_{|u|>\de}  |f(u)|^q \\
& \le \e^q \int_0^{2\pi} |u|^q + C(\de) \int_0^{2\pi} |u|^{q\theta} \\[0.25pc]
& \le C (\e^q \| u\|^q + C(\de) \| u\|^{q\theta}).
\end{align*}
Now choose an arbitrary $\varphi\in X$. We have
\begin{align*}
\frac{|G(u)\varphi|}{\| u\|} & \le \frac{1}{\| u\|} \int_0^{2\pi} |\tilde f(u)| |\varphi| 
\le \frac{1}{\| u\|} \|\tilde f(u) \|_{L^q(0,2\pi)} \| \varphi \|_{L^{q'}(0,2\pi)}\\[0.25pc]
& \le C \frac{1}{\| u\|} \| f(u)\|_{L^q(0,2\pi)} \| \varphi \| \le C  (\e^q + C(\de) 
\| u\|^{q(\theta-1)})^\frac{1}{q} \| \varphi \|.
\end{align*}
This shows that 
$$
\limsup_{\| u \| \to 0} \frac{1}{\| u \|} \left(\sup_{\varphi \in X} 
\frac{|G(u)\varphi|}{\| \varphi\|} \right) =0,
$$
which shows (b). 

To summarize, we have shown that Theorem 11.4 in \cite{MR845785} can be applied 
to give solutions of \eqref{eq-rabinowitz} with small $\| \cdot \|$ norm and 
$\mu$ close to $(1+k^{2s})^{-1}$ for every $k\ge 1$. The conclusion of the theorem 
is then provided by Lemma \ref{lema-truncamiento}. 
\end{proof}

\bigskip

\begin{remark}\label{remark-crandall-rabinowitz}
{\rm According to the classical result by Crandall and Rabinowitz (\cite{MR0288640}), bifurcation from 
the branch of trivial solutions also takes place from the first eigenvalue $\la_0=0$ of 
\eqref{eq-eigen-periodic}, which is simple.  However, it can be checked that in most cases the 
solutions so obtained are constant, therefore of no interest to us.}
\end{remark}

\bigskip

\section{Variational methods}\label{sect-variacional}
\setcounter{equation}{0}

In this section, we will see how the use of variational methods in the space $X$ leads 
to the proof of Theorem \ref{th-variacional}. According to Theorem \ref{th-equivalencia}, 
it suffices to find weak solutions $u \in X$ of the problem
\begin{equation}\label{eq-potencia-impar-periodico}
\mathcal{L} u = \la u + |u|^{p-1} u \quad \hbox{in } (0,2\pi).
\end{equation}
The proof depends on the sign of the parameter $\la$, so we will consider separately the 
cases $\la<0$ and $\la\ge 0$. {We notice here that in the case that $\lambda>0$ the solutions has to change sign.
When $\la<0$} we can take advantage of the homogeneity of the problem and 
argue in a similar way as in Theorem I.2.1 of \cite{MR2431434}. Namely, we will show that 
the functional 
$$
\widetilde{J}(u)=\frac{1}{2} \| u\|^2 -\frac{\la+1}{2} \int_0^{2\pi} u^2, 
$$
defined on $X$, achieves its minimum in the manifold $M$ given by
$$
M:=\left\{ u\in X:\ \int_0^{2\pi} |u|^{p+1}=1\right\}.
$$
Then by invoking Lagrange multipliers' rule we will check that a suitable 
multiple of a function were the minimum is achieved provides with a weak solution 
of \eqref{eq-potencia-impar-periodico}. In principle, the solution obtained 
could be trivial, i. e. constant, but we will rule out this possibility when 
$|\la|$ is sufficiently large.

\bigskip

\begin{lemma}\label{lema-variacional-negativo}
Let $p>1$ and if $s<\frac{1}{2}$ assume $p<2^*_s-1$. Then there exists $\la_0>0$ such that 
the functional $\widetilde{J}$ achieves its minimum in $M$ at a nonconstant function, provided 
that $\la<-\la_0$. In particular, problem \eqref{eq-potencia-impar} admits a nonconstant positive
$2\pi$-periodic classical solution.
\end{lemma}

\begin{proof}

Let us begin by showing that the functional $\widetilde{J}$ is coercive in $M$.  
In fact the H\"older's inequality implies that
\begin{equation}\label{eq-holder}
\int_0^{2\pi} u^2 \le (2\pi)^\frac{p-1}{p+1} \left( \int_0^{2\pi} |u|^{p+1}\right)^\frac{2}{p+1}
=(2\pi)^\frac{p-1}{p+1},
\end{equation}
for every $u\in M$. If $\lambda\leq -1$ the coercivity is immediate. If $\lambda>-1$
$$
\widetilde{J}(u)\ge \frac{1}{2} \|u\|^2 -\frac{\la+1}{2} (2\pi)^\frac{p-1}{p+1},
$$
and the coercivity is also clear. 

The set $M$ is clearly weakly closed because of the subcriticality assumption on $p$. Indeed, 
if $\{u_n\}_{n=1}^\infty \subset M$ is such that $u_n \hookrightarrow u$, we may pass to 
a subsequence to ensure that $u_n\to u$ in $L^{p+1}(0,2\pi)$ (observe that the assumption $p<2^*_s-1$ is required here but only in the case $s<1/2$ because when $s\geq1/2$ we know that the functions are in $L^{q}(0,2\pi)$, $q>1$). Thus $u$ has $L^{p+1}$ norm 
equal to one which yields $u\in M$.

As for the (sequential) weak lower semicontinuity, if $u_n \hookrightarrow u$ we may pass 
again to a subsequence, to have $u_n\to u$ in $L^{p+1}(0,2\pi)$ and in $L^2(0,2\pi)$. 
Hence by the lower semicontinuity of the norm we have
$$
\liminf_{n\to +\infty} \widetilde{J}(u_n) \ge \frac{1}{2} \|u\| ^2 - \frac{\la+1}{2} \int_0^{2\pi}
u^2 =\widetilde{J}(u),
$$
as was to be proved, that is, $\widetilde{J}$ is lower semicontinuous.

Thus by Theorem I.1.2 in \cite{MR2431434} the functional $\widetilde{J}$ achieves its minimum at some 
point $v \in M$. Notice that since
$$||v(x)|-|v(y)||\leq |v(x)-v(y)|,\quad x,y\in \mathbb{R},$$
we have $\widetilde{J}(|v|)\leq \widetilde{J}(v)$ so we may assume $v \geq 0$ .

By the Lagrange multiplier rule, there exists $\mu\in \R$ such that
$$
\left<v,\varphi\right> - (\la+1) \int_0^{2\pi} v \varphi = \mu \int_0^{2\pi} |v|^{p-1}v \varphi,
$$
for every $\varphi \in X$. Taking $\varphi =v$ we see that, since $\lambda<0$,
$$
\mu = \| v\|^2 - (\la+1) \int_0^{2\pi} v^2 \ge -\la \int_0^{2\pi} v^2 >0.
$$
{ Setting $u=\mu^\frac{1}{p-1} v \geq 0$, we see that $u$ is a weak solution 
of \eqref{eq-potencia-impar-periodico} which is not identically zero and by the strong minimum principle $u>0$.} 

To conclude the proof, we need to show that $u$ is not constant, which is equivalent to say 
that $u\not \equiv  (2\pi)^{-\frac{1}{p+1}}$. We will check that this is true provided that $|\la|$ is 
large enough. It suffices to prove the existence of some $u_0\in M$ such that
$$
\| u_0 \|^2 - (\la+1)\int_0^{2\pi} u_0^2 < -\la (2\pi)^\frac{p-1}{p+1}.
$$
Notice that this is equivalent to 
\begin{equation}\label{eq-aux-10}
-\frac{1}{\la} B(u_0) + \int_0^{2\pi} u_0^2 <(2\pi)^\frac{p-1}{p+1},
\end{equation}
where we are temporarily denoting
$$
B(u_0)= \frac{1}{2} \int_0^{2\pi} \hspace{-2mm} \int_0^{2\pi} (u_0(x)-u_0(y))^2 H(x-y)dy dx.
$$
Finally, let us remark that if $u_0$ is not constant then it is well-known that 
the inequality in \eqref{eq-holder} (just replace $u$ by $u_0$) is strict. Thus choosing a nonconstant $u_0\in M$ and 
then letting $|\la|$ be large enough we see that \eqref{eq-aux-10} holds, as was to be proved.
\end{proof}

\bigskip

Next we turn to the case $\la>0$. The procedure followed in Lemma \ref{lema-variacional-negativo} 
is useless, since the infimum of the functional $\widetilde{J}$ in the manifold $M$ is negative. 
Thus we resort to the `direct' functional 
\begin{equation}\label{eq-def-J}
J(u)= \frac{1}{2} \| u\|^2 -\frac{\la+1}{2} \int_0^{2\pi} u^2 - \frac{1}{p+1} \int_0^{2\pi} |u|^{p+1},
\end{equation}
and show that it has a critical point by means of the well-known Linking Theorem. 
The argument is an adaptation of the one given in Chapter 5 of \cite{MR845785} (see also 
\cite{MR3002745}).

\begin{lemma}\label{lema-variacional-positivo}
Let $p>1$ and if $s<\frac{1}{2}$ assume that $p<2^*_s-1$. Then for every $\la>0$ the 
functional $J$ defined in \eqref{eq-def-J} admits a critical point $u\in X$, which is a 
nonconstant {$2\pi$ periodic classical {sign-changing} solution of \eqref{eq-potencia-impar}}.
\end{lemma}

\begin{proof}
For details on the linking theorem, see for instance Theorem 5.3 in \cite{MR845785}. In order 
to check the needed hypotheses, we divide the proof in several steps.

\medskip

\noindent {\it Step 1. The functional $J$ verifies the Palais-Smale condition}.

\smallskip

Let $\{u_n\}_{n=1}^\infty{\subseteq X}$ be a Palais-Smale sequence, that is, $J(u_n)\to c\in \R$ 
and $J'(u_n)\to 0$. Choose $\gamma \in (\frac{1}{p+1},\frac{1}{2})$. Then 
\begin{align*}
{c+o(1)=}J(u_n)-\gamma J'(u_n)u_n &= \left(\frac{1}{2}-\gamma\right) \| {u_n}\|^2 - (\la+1)\left(\frac{1}{2}-\gamma\right) 
\int_0^{2\pi} {u^2_n} \\
& + \left(\gamma-\frac{1}{p+1}\right) \int_0^{2\pi} |{u_n}|^{p+1} 
\end{align*}
Now choosing $\e>0$ small and using H\"older's {and Young's inequalities we get}
$$
(\la+1)\left(\frac{1}{2}-\gamma\right) \int_0^{2\pi}{u^2_n} \le C \left(\int_0^{2\pi} |{u_n}|^{p+1}\right)^\frac{2}{p+1} 
\le \e \int_0^{2\pi} |{u_n}|^{p+1} + D{(\e, p)},
$$
where $C$ and $D$ are positive constants. Hence we deduce that
\begin{align*}
c+o(1)=J(u_n)-\gamma J'(u_n)u_n & \ge \left(\frac{1}{2}-\gamma\right) \| {u_n}\|^2 
+ \left(\gamma-\frac{1}{p+1}-\e \right) \int_0^{2\pi} |{u_n}|^{p+1} -D\\
& \ge \left(\frac{1}{2}-\gamma\right) \| {u_n}\|^2 -D,
\end{align*}
if $\e$ is chosen small enough. 
The previous inequality clearly entails the boundedness of $\{u_n\}_{n=1}^\infty$. By passing to a subsequence, 
we may assume that $u_n\rightharpoonup u$ weakly in $X$, while $u_n\to u$ 
strongly in $L^2(0,2\pi)$ and $L^{p+1}(0,2\pi)$. Since $J'(u_n)\to 0$, we also
deduce $J'(u_n)(u_n-u)\to 0$. But
\begin{align*}
o(1)=J'(u_n)(u_n-u) & = \left< u_n,u_n-u\right> - (\la+1) \int_0^{2\pi} u_n(u_n-u) \\
& \ \  - \int_0^{2\pi} |u_n|^{p-1}u_n (u_n-u)\\[.25pc]
& = \|u_n\|^2 - \left<u_n,u\right> + o(1) = \|u_n\|^2 - \|u\|^2 + o(1).
\end{align*}
Thus $\|u_n\|\to \|u\|$, which implies that $u_n\to u$ in $X$. This shows that 
$J$ verifies the Palais-Smale condition.

\medskip

\noindent {\it Step 2. Let $k\in \mathbb{N}$. Since $\lambda>0$ we can suppose that $(k-1)^{2s} \le \la < k^{2s}$ for some $k\in\mathbb{N}$, that is, $\lambda\in [\lambda_{k-1}, \lambda_k)$ where $\lambda_k$ is the $k$-eigenvalue of $\mathcal{L}$. For that $k$, let us consider $F_k$ the set given in Lemma \ref{lema-eigen-variacional}. Then there exist $r>0$, $\beta>0$ such 
that $J(u)\ge \beta$ if $u\in F_k$ and $\|u\|= r$}.

\smallskip

First of all, observe that, by \eqref{eq-caract-variacional} in Lemma \ref{lema-eigen-variacional} 
we have
$$
\| u\|^2 \ge (k^{2s}+1) \int_0^{2\pi} u^2, \quad \hbox{for every } u\in F_k.
$$
Therefore, it is immediate that for every $u\in F_k$ with $\|  u\|=r$:
\begin{align*}
J(u) & \ge \frac{k^{2s}-\la}{2(k^{2s}+1)} \| u \|^2 - \frac{1}{p+1} \int_0^{2\pi} |u|^{p+1}\\
& \ge \frac{k^{2s}-\la}{2(k^{2s}+1)} \| u \|^2 - C \| u\|^{p+1} \\
& = \frac{k^{2s}-\la}{2(k^{2s}+1)} r^2 - C r ^{p+1} :=\beta,
\end{align*}
and $\beta$ can be chosen to be positive just by choosing $r$ small enough.

\medskip

\noindent {\it Step 3. Let $k$ be as in Step 2 and denote  
$$
E_k={\rm span} \left\{ \cos(jx),\sin(jx) \right\}_{j=0}^{k-1}. 
$$
Then $J(u)\le 0$ for every $u\in E_k$.}

\smallskip

Choose $u\in E_k$, so that it can be written in the form
$$
u=\frac{a_0}{2} + \sum_{j=1}^{k-1} a_j \cos(jx)+b_j \sin(jx).
$$
Then, using  \eqref{eq-ortogonalidad} we have
\begin{align*}
\| u\|^2 & =\pi \left( \frac{a_0^2 }{2} + \sum_{j=1}^{k-1} (j^{2s}+1)(a_j^2+b_j^2) \right)\\
& \le \pi ((k-1)^{2s}+1)\left( \frac{a_0^2 }{2} + \sum_{j=1}^{k-1} (j^{2s}+1)(a_j^2+b_j^2) \right)\\
& = ((k-1)^{2s}+1)\int_0^{2\pi} u^2,\quad {u\in E_k}.
\end{align*}
It follows that
$$
J(u) \le \frac{1}{2} \| u\|^2 -\frac{\la+1}{2} \int_0^{2\pi} u^2 
\le \frac{(k-1)^{2s}-\la}{2}\int_0^{2\pi} u^2 \le 0
$$
for every $u\in E_k$.

\medskip

\noindent {\it Step 4. Let $Y$ be any finite-dimensional subspace of $X$. Then there exists $R>{r}$ 
such that $J(u)< 0$ if $u\in Y$, $\|u\| \ge R$. }
 
\smallskip

Since all norms defined on a finite dimensional vector space are equivalent, 
there exists a constant $C>0$ such that 
$$
\frac{1}{p+1} \int_0^{2\pi} |u|^{p+1} \ge C \| u\| ^{p+1}, \quad \hbox{for every } u\in Y.
$$
Thus, using that $\la>0$, 
$$
J(u) \le \frac{1}{2} \|u\|^2 - C \| u\|^{p+1}<0,
$$
provided that $\|u\|\ge R$ and $R$ is large enough.

\medskip

\noindent {\it Step 5. Conclusion.} We can use Theorem 5.3 in \cite{MR845785} (see also Remark 5.5 (iii) there) 
to conclude that $J$ admits a critical point $u\in X\setminus \{0\}$. Then $u$ will be a weak 
solution of \eqref{eq-potencia-impar-periodico} which is not identically zero. Observe that problem 
\eqref{eq-potencia-impar-periodico} does not admit constant solutions aside the trivial one when $\la>0$, 
so that $u$ is a nonconstant solution of \eqref{eq-potencia-impar-periodico} and 
by Theorem \ref{th-equivalencia} a nonconstant $2\pi$-periodic solution of \eqref{eq-potencia-impar}. 
The proof is concluded. 
\end{proof}

\bigskip
\begin{proof}[Proof of Theorem \ref{th-variacional}]
It is clear that the desired conclusion follows by using Lemma \ref{lema-variacional-negativo} and Lemma \ref{lema-variacional-positivo}
\end{proof}

\begin{remark}\label{remark-variacional-general}
{\rm Weak solutions of the slightly more general problem
$$
\mathcal{L} u = \la u + f(u) \quad \hbox{in } (0,2\pi)
$$
can also be obtained with a minor variation of the proof of Lemma \ref{lema-variacional-positivo}, 
provided that $f$ is a $C^1$ function which verifies some natural growth restrictions (see 
\cite{MR845785} and \cite{MR3002745}). Namely:

\begin{itemize}

\item There exist $p>1$ and $C>0$ such that $|f(t)|\le C(1+|t|)^p$, where $p<2^*_s-1$ if $s<\frac{1}{2}$;

\item $f'(0)=0$ and $tf(t)\ge 0$ if $t\in \R$;

\item There exists $\mu>2$ such that $0<\mu \ds \int_0^t f(\tau)d\tau \le tf(t)$ for large $|t|$.
\end{itemize}
}\end{remark}

\bigskip

\section{Some examples}\label{sect-ejemplos}
\setcounter{equation}{0}

In this final section we are focusing our attention on some particular 
examples. Unlike their local counterparts, it is difficult to classify 
all their solutions, so we only obtain existence of periodic solutions.

To begin with, let $p>1$ and consider 
\begin{equation}\label{eq-ejemplo-par}
(-\De)^s u = u+ |u|^p \quad \hbox{in } \R.
\end{equation}
When $s=1$, the phase-space for this problem shows that there exist 
periodic solutions with small amplitude, but the norms of all periodic solutions is 
uniformly bounded. The periods of the solutions range from $2\pi$ to infinity. For the nonlocal 
version \eqref{eq-ejemplo-par} we obtain:

\begin{corollary}\label{coro-1}
Let $p>1$. Then problem \eqref{eq-ejemplo-par} admits {at least two} classical periodic solutions {$u_i$, $i=1,\, 2$} with 
small amplitude and minimal period close to $2\pi$. In addition if $p<2^*_s-1$ when $s<1/2$ then there exists $C>0$ such that 
\begin{equation}\label{eq-cotas}
\| u\|_{L^\infty(\R)} \le C
\end{equation}
for every periodic solution of \eqref{eq-ejemplo-par}
\end{corollary}

\begin{proof}[Proof of Corollary \ref{coro-1}]
Consider the problem
$$
(-\De)^s v = \la v + |v|^p \quad \hbox{in } \R.
$$
By Theorem \ref{th-bifurcacion} with $k=1$, there exist {at least two} solutions {$(\lambda_i, v_i)_{i=1,2}$} of small amplitude 
for some values of {$\lambda_i$, $i=1,\, 2$,} close to $\frac{1}{2}$ and whose minimal periods are 
close to $2^\frac{2s+1}{2s} \pi$. Then it is easily seen that 
$$
{u_i}(x)= {\la_i}^{-\frac{1}{p-1}} {v_i}({\la_i}^{-\frac{1}{2s}} x), \quad x\in \R,\,  {i=1,\, 2,}
$$
are solutions of \eqref{eq-ejemplo-par} with small amplitude and periods close 
to $2\pi$.

To show the universal boundedness of the solutions we will use the well-known blow-up 
method of \cite{MR619749}. But first we reduce the problem to deal only with positive 
solutions. If $u$ is a periodic solution of \eqref{eq-ejemplo-par} and $x_0\in \R$ is a point 
where its minimum is achieved we have $(-\De)^s u(x_0)<0$. Then $u(x_0)+|u(x_0)| ^p<0$, which 
yields $u(x_0)>-1$. Thus $u>-1$ in $\R$. 

Let $w= u+1$, which is a positive solution of the problem
\begin{equation}\label{eq-problema-positivo}
(-\De)^s w = w-1+ |w-1|^p \quad \hbox{in } \R.
\end{equation}
The proof will be concluded if we show that \eqref{eq-cotas} holds for solutions of 
\eqref{eq-problema-positivo}. Thus assume it does not hold and let $\{w_n\}_{n=1}^\infty$ 
be a sequence of periodic solutions of \eqref{eq-problema-positivo} such that 
$M_n:=\| w_n\|_{L^\infty(\R)}\to +\infty$. Choose points $x_n\in \R$ where the functions $w_n$ attain 
their maxima and introduce the functions:
$$
z_n(x)= \frac{w_n(M_n^{-\frac{2s}{p-1}}x+x_n)}{M_n}, \quad x\in \R.
$$
It is not hard to check that $z_n$ verifies the equation
$$
(-\De)^s z_n = M_n^{-\frac{1}{p-1}} z_n - M_n ^{-\frac{p}{p-1}} + |z_n-M_n^{-1}|^p \quad \hbox{in }\R,
$$
while $0<z_n\le 1$ and $z_n(0)=1$. Now we can argue in the usual way with the use of interior regularity 
(cf. \cite{MR2270163} or \cite{MR2494809}) to deduce the existence of a subsequence and a function 
$w\in C^\infty(\R)$ such that $0<w\le 1$, $w(0)=1$ and 
$$
(-\De)^s w= w^p \quad \hbox{in } \R.
$$
However, this contradicts Theorem 1.2 in \cite{MR2739791} when $s\ge \frac{1}{2}$ and Theorem 4 in 
\cite{2014arXiv1401.7402Z} (see also \cite{MR2200258}) if $s<\frac{1}{2}$. We remark that all these 
results are still valid when the equations are considered in one dimension.

This contradiction shows that \eqref{eq-cotas} must be true, and the proof is concluded.
\end{proof}

\medskip

Now we will consider two problems where the nonlinearity in \eqref{eq-ejemplo-par} is replaced 
by an odd power. These problems are 
\begin{equation}\label{eq-ejemplo-impar-1}
(-\De)^s u = u+ |u|^{p-1}u \quad \hbox{in } \R
\end{equation}
and
\begin{equation}\label{eq-ejemplo-impar-2}
(-\De)^s u = -u+ |u|^{p-1}u \quad \hbox{in } \R.
\end{equation}
A similar reasoning as above, using Theorem \ref{th-variacional} instead of Theorem \ref{th-bifurcacion}, 
gives the following corollaries:

\begin{corollary}\label{coro-2}
Assume $p>1$ and $p<2^*_s-1$ in the case $s<1/2$. Then problem \eqref{eq-ejemplo-impar-1}
admits classical {sign-changing} periodic solutions for every period. 
\end{corollary}

\begin{corollary}\label{coro-3}
Assume $p>1$ and $p<2^*_s-1$ in the case $s<1/2$. Then problem \eqref{eq-ejemplo-impar-2}
admits classical {positive} (and negative) periodic solutions with large periods. 
\end{corollary}

\begin{remark}\label{rem}{\rm 
i) Two homoclinic solution of \eqref{eq-ejemplo-impar-2}, can be seen to exist by
the existence and decay estimates established in Proposition 1.1 of \cite{MR3070568}. In fact the existence of a positive $v$ homoclinic solution of $(-\De)^s v+v - v^p=0$  in $\R$ comes from \cite{MR3070568} and the negative just considering the change $u:=-v< 0$ that clearly satisfies $(-\De)^s u+u - |u|^{p-1}u=0$  in $\R$ . Since $u$ is negative we have  $(-\De)^s u+u + |u|^{p}=0$  in $\R$. \\
ii) With the same argument positive periodic solutions of \eqref{eq-ejemplo-impar-2} with large periods (given by the above corollary) give negative periodic solutions of  $(-\De)^s u+u + |u|^{p}=0$  in $\R$ with large periods. Also the negative homoclinic solution for this problem is obtained. {For the case $s=1$ this equation corresponds to $u''=u+|u|^p$ with phase portrait as in Figure 1 (left).}}
\end{remark}

It is worthy of mention that in the local case $s=1$ all solutions of 
\eqref{eq-ejemplo-impar-1} are periodic, and 
there exist solutions of all periods. See the phase portrait in Figure 1 (right).

\medskip
\begin{center}

\includegraphics[width=9cm]{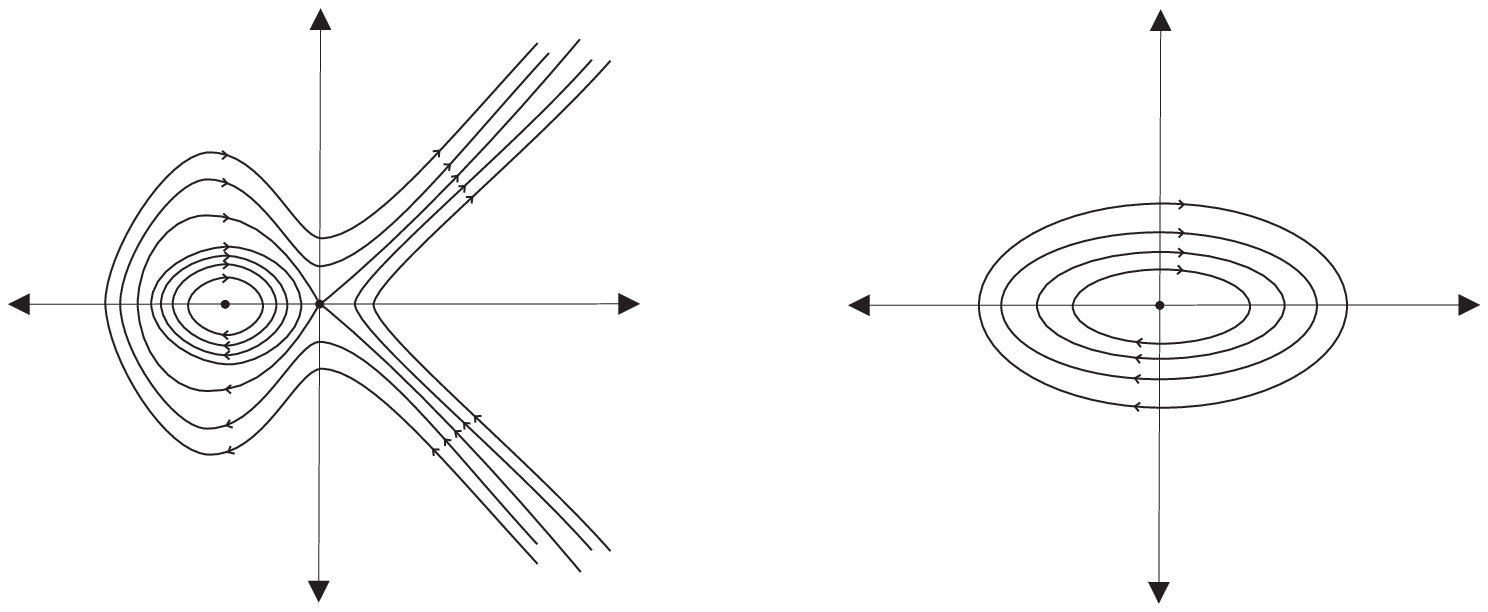}

\medskip
{\sc Figure 1.} 

\end{center}

\bigskip

To close this section, we will briefly consider the stationary version of the 
well-known Benjamin-Ono equation, given by 
\begin{equation}\label{eq-benjamin-ono}
u_x - 2uu_x +(-\De)^s u_x=0 \quad \hbox{in } \R.
\end{equation}
The Benjamin-Ono equation was introduced for $s=1/2$ in \cite{MRninguna} and \cite{MR0398275} in the 
attempt to model one-dimensional internal waves in deep water, and has been well-studied since. See also 
\cite{JCFNthesis}.

We will show that as a consequence of Corollary \ref{coro-1} {and Corollary \ref{coro-3}}, periodic solutions of 
\eqref{eq-benjamin-ono} can be constructed.

\begin{corollary}\label{coro-benjamin-ono}
{Problem \eqref{eq-benjamin-ono} for $s>1/6$ admits three classical solution: A periodic solutions with minimal period close to $2\pi$ and amplitud close to $1$, a periodic {positive} solutions with large periods  and a positive homoclinic solution}
\end{corollary}

\begin{proof}
Let $v$ be a classical periodic solution {with small amplitude and minimal period close to $2\pi$} of \eqref{eq-ejemplo-par} with $s>\frac{1}{6}$ and
$p:=2<2^*_s-1$. Then $v\in C^\infty(\R)$ by standard regularity. If we let $u= v+1$, it 
is easily seen that $u$ is a periodic classical solution of
\begin{equation}\label{dg}
(-\De)^s u = -u+u^2 \quad \hbox{in } \R.
\end{equation}
Moreover, since $u\in C^\infty(\R)$, we can differentiate this equation to obtain that 
$u$ is a periodic solution of \eqref{eq-benjamin-ono}. 

On the other hand by Corollary \ref{coro-3}, when $s>\frac{1}{6}$, there exists $w> 0$ a classical periodic solution of \eqref{dg} with large period and the positive homoclinic solution $\widetilde{w}> 0$ associated to this problem (see Remark \ref{rem}). Again differentiate the equation satisfies by these $C^\infty(\R)$ solutions we find two more kinds of solutions to \eqref{eq-benjamin-ono}. 
\end{proof}

\bigskip


\noindent {\bf Acknowledgements.} 
	All authors were partially supported by Ministerio de Eco\-no\-m\'ia y 
	Competitividad under grant MTM2014-52822-P (Spain). 
	B. B. was partially supported by a MEC-Juan de la Cierva postdoctoral 
	fellowship number  FJCI-2014-20504 (Spain). 
	A. Q. was partially supported by Fondecyt Grant No. 1151180 Programa Basal, 
	CMM. U. de Chile and Millennium Nucleus Center for Analysis of 
	PDE NC130017. 
	

\bibliographystyle{abbrv}

\bibliography{Bibliografia}

\end{document}